\documentclass{amsart}
\usepackage{amsmath, amsthm, amssymb, booktabs, multirow, graphicx, booktabs,tikz}

\usepackage[colorlinks]{hyperref}
\usepackage{xcolor}
\definecolor{lcred}{HTML}{973B3A}
\definecolor{lcbeige}{HTML}{D9BFAC}
\definecolor{lcblue}{HTML}{4F6AA5}

\counterwithin{figure}{section}

\counterwithin{table}{section}

\newtheorem{defin}{Definition}[section]

\newtheorem{theorem}[defin]{Theorem}
\newtheorem{lemma}[defin]{Lemma}

\newcommand{\R}{\mathbb{R}}

\newcommand{\Z}{\mathbb{Z}}

\newcommand{\N}{\mathbb{N}}

\DeclareMathOperator{\Cayley}{Cayley}
\DeclareMathOperator{\Vor}{Vor}
\DeclareMathOperator{\Del}{Del}
\DeclareMathOperator{\GL}{GL}
\DeclareMathOperator{\Ort}{O}
\DeclareMathOperator{\cone}{cone}

\begin{document}

\title{The chromatic number of 4-dimensional lattices}

\address{F.~Vallentin, S.~Weißbach, M.C.~Zimmermann, Department Mathematik/Informatik, Abteilung
  Mathematik, Universit\"at zu K\"oln, Weyertal~86--90, 50931 K\"oln,
  Germany\vspace*{2ex}}

\author{Frank Vallentin}
\email{frank.vallentin@uni-koeln.de}

\author{Stephen Weißbach}
\email{stepwb71@gmail.com}

\author{Marc Christian Zimmermann}
\email{marc.christian.zimmermann@gmail.com}

\date{November 8, 2024}

\begin{abstract}
  The chromatic number of a lattice in $n$-dimensional Euclidean space
  is defined as the chromatic number of its Voronoi graph. The Voronoi graph is
  the Cayley graph on the lattice having the strict Voronoi vectors as
  generators. In this paper we determine the chromatic number of all
  $4$-dimensional lattices.

  To achieve this we use the known classification of $52$ parallelohedra in
  dimension~$4$. These $52$ geometric types yield $16$ combinatorial
  types of relevant Voronoi graphs. We discuss a systematic approach
  to checking for isomorphism of Cayley graphs of lattices.
  
  Lower bounds for the chromatic number are obtained from choosing appropriate
  small finite induced subgraphs of the Voronoi graphs. Matching
  upper bounds are derived from periodic colorings. To determine the chromatic
  numbers of these finite graphs, we employ a SAT solver.
\end{abstract}

\maketitle

\markboth{F. Vallentin, S. Weißbach, and M.C. Zimmermann}
{The chromatic number of 4-dimensional lattices}

\centerline{\large \em Dedicated to Rudolf Scharlau in occasion of his
  72nd birthday}

\section{Introduction}

Let $\Lambda \subseteq \R^n$ be an $n$-dimensional lattice in
$n$-dimensional Euclidean space. The \emph{chromatic number} of
$\Lambda$ is the smallest number of colors one needs to color lattice
translates of the Voronoi cell
\[
  V(\Lambda) = \{ x \in \R^n : \|x\| \leq \|x - v\| \text{ for
    all } v \in \Lambda\}
\]
so that distinct lattice translates $v + V(\Lambda)$ and
$w + V(\Lambda)$, with $v \neq w$, receive different colors whenever
the intersection $(v + V(\Lambda)) \cap (w + V(\Lambda))$ is a common
$(n-1)$-dimensional face (a facet) of the two polytopes
$v + V(\Lambda)$ and $w + V(\Lambda)$. This chromatic number is
denoted by $\chi(\Lambda)$.  

\smallskip

The chromatic number of lattices was first considered by Dutour
Sikiri\'c, Madore, Moustrou, Vallentin in \cite{Dutour-2021}.

\smallskip

Equivalently, the chromatic number of a lattice $\Lambda$ is the
chromatic number of the (infinite) Cayley graph on the additive group
$\Lambda$ with strict Voronoi vectors as its set of generators. The \emph{Voronoi vectors} of $\Lambda$, are the differences $v-w$ such that the intersection of the distinct translates $v+ V(\Lambda)$ and $w + V(\Lambda)$ is a common face; but not necessarily a common facet. A Voronoi vector as above is a \emph{strict Voronoi vector} if this common intersection is a facet, we denote the set of such vectors by $\Vor(\Lambda)$.
Thus,
\[
\chi(\Lambda) = \chi(\Cayley(\Lambda, \Vor(\Lambda))),
\]
where the vertices of the Cayley graph are all elements of $\Lambda$
and two distinct vertices $v, w$ are adjacent whenever the difference
$v - w$ lies in~$\Vor(\Lambda)$. We call the Cayley graph
$\Cayley(\Lambda, \Vor(\Lambda))$ the \emph{Voronoi graph} of the
lattice $\Lambda$.

\smallskip

Many natural questions regarding $\chi(\Lambda)$ remain 
unanswered.

\smallskip

For instance, it is not known whether $\chi(\Lambda)$ is
computable. By the compactness theorem of de Bruijn and Erd\H{os}
\cite{deBruijn-1951} we only know that the decision problem ``Is
$\chi(\Lambda) \leq k$?'' is semidecidable because $\chi(\Lambda)$ is
equal to the largest chromatic number of all finite, induced subgraphs
of $\Lambda$. This means that there is an algorithm which terminates with a positive answer after finitely many steps on every input $\Lambda$ with $\chi(\Lambda) \leq k$. On inputs $\Lambda$ with $\chi(\Lambda) > k$, the algorithm either terminates with a negative answer or it runs in an infinite loop.

\smallskip

Directly related is the question of whether one can always find a periodic coloring with $\chi(\Lambda)$ colors. If so, and
if we knew a priori a bound on the period, then this would immediately
yield an algorithm for computing $\chi(\Lambda)$.

\smallskip

Also the extremal question of maximizing $\chi(\Lambda)$ among all
$n$-dimensional lattices $\Lambda$ is widely open. In
\cite{Dutour-2021} it is shown, using a combination of the
Kabatiansky-Levenshtein upper bound and the Minkowski-Hlawka lower
bound for sphere packings, that this maximum grows exponentially for
generic lattices. In particular, there are $n$-dimensional lattices
with $\chi(\Lambda_n) \geq 2 \cdot 2^{(0.0990\ldots - o(1))n}$. By an
easy greedy argument, there is also an exponential upper bound
$\chi(\Lambda) \leq 2^n$ for every $n$-dimensional lattice
$\Lambda$. Currently, an efficient way (polynomial time in the dimension) to construct lattices with
exponential chromatic number is not known.

\smallskip

Turning to low dimensions, in \cite{Dutour-2021} the two- and
three-dimensional cases are settled completely. There are two cases in
dimension two: the square lattice and the hexagonal lattice, with 
chromatic number two and three, respectively. There are five cases in
dimension three. The Voronoi cells of these five cases are the cube,
the hexagonal prism, the rhombic dodecahedron, the elongated
dodecahedron, and the truncated octahedron. The corresponding
chromatic numbers are two, three, and four (three times).

\medskip

In this paper we completely settle the four-dimensional case.

\begin{theorem}
  \label{thm:main}
  If $\Lambda$ is a four-dimensional lattice, then its Voronoi graph is
  isomorphic to one of the 16 cases given in Table
  \ref{table:Cayley-graphs}.\footnote{After the first version of our paper was submitted to the \texttt{arXiv.org} e-print archive, Igor Baburin contacted us and reported that he arrived at the same number a few years ago.} The 16 pairwise non-isomorphic graphs
  have chromatic numbers as given in Table
  \ref{table:Cayley-graphs-computational-data}.
\end{theorem}

The proof consists of two parts.

\smallskip

In Section~\ref{sec:classification} we show that, up to graph
isomorphism, there are $16$ different Voronoi graphs to consider. This
classification is a consequence of a classical result in the geometry
of numbers, essentially due to Delaunay \cite{Delaunay-1929}: In
dimension $4$, there are $52$ types of Voronoi cells. We prove that
these $52$ types only yield $16$ non-isomorphic Voronoi graphs. In
Section~\ref{sec:classification} we describe these classifications in
detail. In addition, we present an argument that shows how to
algorithmically test the isomorphism of Cayley graphs of lattices. Since
these are infinite graphs, it is not a priori clear that such an
algorithm needs to exist.

\smallskip

In Section~\ref{sec:bounds} we consider each of the $16$ cases
separately. For every case we first determine lower bounds for
$\chi(\Lambda)$ by computing the chromatic number of a finite, induced
subgraph of the Voronoi graph $\Cayley(\Lambda, \Vor(\Lambda))$. Every
natural number $d$ defines a finite, induced subgraph of the Voronoi
graph by using the vertices which are at most (graph) distance~$d$
from the origin. We denote this graph by
$\Cayley(\Lambda, \Vor(\Lambda))_d$. Then, for each case, we determine
matching upper bounds by constructing explicit periodic colorings. The
computations for the lower bounds and for the upper bounds both use
SAT solvers---algorithms to certify whether a given Boolean formula in
conjunctive normal form has a satisfying assignment or not.

The proof of Theorem~\ref{thm:main} involves numerous computations.
To ensure transparency and enable the reader to verify our
calculations, we have provided a program, written using the computer
algebra system MAGMA \cite{MAGMA}. This MAGMA program is available as
an ancillary file from the \texttt{arXiv.org} e-print archive. We
discuss the MAGMA program at the end of Section~\ref{sec:bounds}
once all the needed objects have been introduced.

\smallskip

In Section~\ref{sec:dimension-5} we collect some questions for further
research.

\section{Voronoi graphs of lattices in four dimensions}
\label{sec:classification}

In this section we show that there are, up to graph isomorphism, $16$
Voronoi graphs $\Cayley(\Lambda, \Vor(\Lambda))$ where $\Lambda$ is a
four-dimensional lattice.

\smallskip

The starting point for this classification of Voronoi graphs is a
classical classification result by Voronoi~\cite{Voronoi-1908} concerning the different Delaunay subdivisions of $4$ dimensional lattices. 

\smallskip

Note that there is another possible starting point with the classification of iso-edge domains (also called $C$-type domains). For more information on the concept of iso-edge domains we refer to Ryshkov and Baranovskii \cite{Ryshkov-1976} and to Dutour Sikiri\'c and Kummer \cite{Dutour-2022}. 
In general we now have $3$ equivalence relations regarding lattices, by identifying lattices according to either their Voronoi graph, their iso-edge domain, or their Delaunay subdivision. 
It is a priori clear that these relations become finer with Delaunay subdivisions being the finest one and Voroni graphs being the coarsest one. 
In dimension $4$ we have, up to equivalence, $51$ inequivalent iso-edge domains and $52$ inequivalent Delaunay subdivisions, so for our approach there is not much of a difference in using one of the two. 
We prefer to work with Delaunay subdivisions, as they are more common in the literature. 
Note that our classification of Voroni graphs in dimension $4$ yields that there are $16$ such, up to equivalence, so their number is strictly smaller than the number of iso-edge domains.

\smallskip

We now return to the discussion of Delaunay subdivisions.
To state Voronoi's result it is convenient to work with positive definite
quadratic forms instead of lattices.

By $\mathcal{S}^n_{++}$ we denote the open set of positive definite
quadratic forms, which we identify by positive definite matrices
through $x \mapsto Q[x] = x^{\sf T} Q x$. The dictionary between
lattices in Euclidean space and positive definite quadratic forms can
be compactly expressed by the double coset
$\Ort_n(\R) \backslash \GL_n(\R) / \GL_n(\Z)$, representing Euclidean
lattices up to orthogonal transformations and lattice basis
transformations, which is equal to $\mathcal{S}^n_{++} / \GL_n(\Z)$,
positive definite quadratic forms up to the conjugation action by
$\GL_n(\Z)$, as $\Ort_n(\R) \backslash \GL_n(\R) = \mathcal{S}^n_{++}$.

Below we define the Delaunay subdivision of a positive definite
quadratic form. This polytopal subdivision is geometrically dual to
the subdivision of $\R^n$ given by lattice translates of Voronoi
cells.

A polytope $P \subseteq \R^n$ with only integral vertices is called a
\emph{Delaunay polytope} of $Q \in \mathcal{S}^n_{++}$ if there exists
$c \in \R^n$ and $r > 0$ such that
\[
  Q[x - c] \geq r \text{ for all $x \in \Z^n$ } \; \text{ and
  } \;
  Q[x-c] = r \text{ for all vertices $x$ of $P$}.
\]
The set
\[
  \Del(Q) = \{ P : \text{$P$ is a Delaunay polytope of $Q$}\}
\]
gives the \emph{Delaunay subdivision} of $Q$. If all occurring Delaunay
polytopes are simplices we speak of a \emph{Delaunay triangulation}.

Now we shall explain the connection between Delaunay subdivisions and
Voronoi graphs of lattices. Let $\Lambda$ be a lattice with lattice
basis $B \in \GL_n(\R)$. When we take the linear image, under $B$, of
the vertex-edge graph of the Delaunay subdivision of
$Q = B^{\sf T} B$, then we exactly get the Voronoi graph
$\Cayley(\Lambda, \Vor(\Lambda))$. So classifying Delaunay
subdivisions up to the action of $\GL_n(\Z)$ provides a finer
classification than the classification of Voronoi graphs up to the
graph isomorphism.

In his second memoir \cite{Voronoi-1908}, Voronoi showed that up to
the action of $\GL_4(\Z)$ there are exactly three different Delaunay
triangulations of quaternary positive definite quadratic forms. We
recall some details of his result in
Section~\ref{sec:classification-triangulations}.

\smallskip

Every Delaunay subdivision $\mathcal{D}$ is refined by some Delaunay
triangulation $\mathcal{D}'$, meaning that all elements of
$\mathcal{D}'$ are simplices and for every simplex
$D' \in \mathcal{D}'$ there is a polytope $D \in \mathcal{D}$ so that
$D' \subseteq D$. So classifying Delaunay subdivisions amounts to
finding all possible Delaunay subdivisions which coarsen some Delaunay
triangulation. The classification of Delaunay subdivisions of
quaternary positive definite quadratic forms is due to Delaunay
\cite{Delaunay-1929}. Interestingly, he used the dual approach of
classifying four-dimensional parallelohedra.

\emph{Parallelohedra} are polytopes which tile space by lattice
translates. For dimension $4$, Delaunay showed that parallelohedra are
affine images of Voronoi cells of lattices. In general, it is not
known whether every parallelohedron arises as an affine image of a
Voronoi cell. This question is called Voronoi's conjecture.

Delaunay claimed that there are $51$ different Delaunay subdivisions
of quaternary positive definite quadratic forms. However, there are
indeed $52$ cases; Delaunay's classification was corrected and
completed by Shtogrin \cite{Shtogrin-1973}. Later Delaunay's result
was reworked, verified, and simplified by Engel \cite{Engel-1992},
Conway \cite{Conway-1997}, Deza, Grishukhin \cite{Deza-2008},
Zhilinskii \cite{Zhilinskii-2015}. We will recall some details of
Conway's classification in
Section~\ref{sec:classification-subdivisions}. More details, and a
complete derivation of Conway's classification, can be found in the
first-named author's PhD thesis \cite{Vallentin-2003}.

\smallskip

In Section~\ref{sec:classification-graphs} we will show that many of
the $52$ cases yield the same or at least isomorphic Voronoi graphs.

\subsection{Classification of Delaunay triangulations}
\label{sec:classification-triangulations}

Here we recall some details of Voronoi's classification of Delaunay
triangulations of quaternary positive definite quadratic forms. We
refer to the book by Sch\"urmann \cite{Schuermann-2009} for a
contemporary exposition of Voronoi's theory.

Let $\mathcal{D}$ be the Delaunay subdivision of some positive
definite quadratic form $Q$, then
\[
  \Delta(\mathcal{D}) = \{Q' \in \mathcal{S}^n_{++} : \mathcal{D} = \Del(Q')\}
\]
is called the \emph{secondary cone} of $\mathcal{D}$. Its topological
closure $\overline{\Delta(\mathcal{D})}$ is a polyhedral cone in the
convex cone of positive semidefinite matrices $\mathcal{S}^n_+$. The
Delaunay subdivision is a triangulation if and only if its secondary
cone is full-dimensional. A Delaunay subdivision $\mathcal{D}'$
refines another subdivision $\mathcal{D}$ if and only if
$\overline{\Delta(\mathcal{D'})} \subseteq
\overline{\Delta(\mathcal{D})}$. The group $\GL_n(\Z)$ acts on the
space of symmetric matrices by conjugation. Under this group action
there are only finitely many orbits of secondary cones.

At the end of his second memoir, Voronoi \cite{Voronoi-1908} computed
the orbits of full-dimensional secondary cones for $n \leq 4$. For
$n \leq 3$ there is only one orbit and for $n = 4$ there are exactly three
orbits. It turns out that in these cases, all full-dimensional secondary cones are simplicial. This is no longer the case for $n \geq 5$. For instance, in \cite{Dutour-2005}, for $n = 6$, a full-dimensional secondary cone (of dimension $\binom{n(n+1)}{2} = 21$) having $7145429$ many extreme rays was found.

Using Voronoi's notation, the three polyhedral domains ${\it\Delta}$,
${\it\Delta'}$, ${\it\Delta''}$ are given by the extremal rays:
\[
\begin{array}{c}
{\it \Delta} = \cone\{R_1, \ldots, R_{10}\},\\
{\it \Delta'} = \cone\{R_1, \ldots, R_4, R_6, \ldots, R_{11}\},\\
 {\it \Delta''} = \cone\{R_1, \ldots, R_4, R_6, \ldots, R_9,R_{11},
  R_{12}\},
\end{array}
\]
where
{\footnotesize
\[
R_1 =
\begin{pmatrix}
1 & 0 & 0 & 0\\
0 & 0 & 0 & 0\\
0 & 0 & 0 & 0\\
0 & 0 & 0 & 0
\end{pmatrix},\;
R_2 = 
\begin{pmatrix}
0 & 0 & 0 & 0\\
0 & 1 & 0 & 0\\
0 & 0 & 0 & 0\\
0 & 0 & 0 & 0
\end{pmatrix},
R_3 = 
\begin{pmatrix}
0 & 0 & 0 & 0\\
0 & 0 & 0 & 0\\
0 & 0 & 1 & 0\\
0 & 0 & 0 & 0
\end{pmatrix},
\]

\[
R_4 =
\begin{pmatrix}
0 & 0 & 0 & 0\\
0 & 0 & 0 & 0\\
0 & 0 & 0 & 0\\
0 & 0 & 0 & 1
\end{pmatrix},\;
R_5 =
\begin{pmatrix}
1 & -1 & 0 & 0\\
-1 & 1 & 0 & 0\\
0 & 0 & 0 & 0\\
0 & 0 & 0 & 0
\end{pmatrix},\;
R_6 =
\begin{pmatrix}
1 & 0 & -1 & 0\\
0 & 0 & 0 & 0\\
-1 & 0 & 1 & 0\\
0 & 0 & 0 & 0
\end{pmatrix},
\]

\[
R_7 = 
\begin{pmatrix}
1 & 0 & 0 & -1\\
0 & 0 & 0 & 0\\
0 & 0 & 0 & 0\\
-1 & 0 & 0 & 1
\end{pmatrix},\;
R_8 = 
\begin{pmatrix}
0 & 0 & 0 & 0\\
0 & 1 & -1 & 0\\
0 & -1 & 1 & 0\\
0 & 0 & 0 & 0
\end{pmatrix},\;
R_9 = 
\begin{pmatrix}
0 & 0 & 0 & 0\\
0 & 1 & 0 & -1\\
0 &  0 & 0 & 0\\
0 & -1 & 0 & 1
\end{pmatrix},
\]

\[
R_{10} = 
\begin{pmatrix}
0 & 0 & 0 & 0\\
0 & 0 & 0 & 0\\
0 & 0 & 1 & -1\\
0 & 0 & -1 & 1
\end{pmatrix},\;
R_{11} =  
\begin{pmatrix}
 4 & 2 & -2 & -2\\
 2 & 4 & -2 & -2\\
-2 & -2 & 4 & 0\\
-2 & -2 & 0 & 4
\end{pmatrix},\;
R_{12}  =  
\begin{pmatrix}
 1 & 1 & -1 & -1\\
 1 & 1 & -1 & -1\\
-1 & -1 & 1 & 1\\
-1 & -1 & 1 & 1
\end{pmatrix}.
\]
}

In our notation this corresponds to
\[
{\it \Delta} = \overline{\Delta(\mathcal{D})}, \;
{\it \Delta'} = \overline{\Delta(\mathcal{D}')}, \;
{\it \Delta''} = \overline{\Delta(\mathcal{D}'')},
\]
where $\mathcal{D}$, $\mathcal{D}'$, $\mathcal{D}''$ are
non-equivalent Delaunay triangulations. In
Figure~\ref{fig:secondary-cone-tessellation} we schematically show how
the rational closure of the positive definite quadratic forms is
tessellated by orbits of the polyhedral domains ${\it \Delta}$,
${\it \Delta'}$, ${\it \Delta''}$.
\begin{figure}[htb]
\begin{center}
\includegraphics[width=10cm]{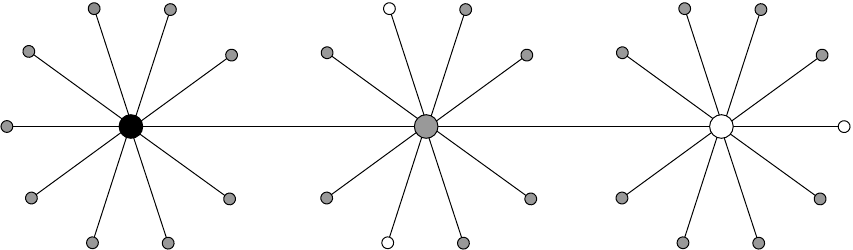}
\end{center}
\caption{A finite part of the tessellation of the rational closure of the positive
  definite quadratic forms into orbits of the polyhedral domains
  ${\it \Delta}$, ${\it \Delta'}$, ${\it \Delta''}$. Black nodes
  correspond to the orbit of ${\it \Delta}$, grey nodes to the orbit
  of ${\it \Delta'}$, and white nodes to the orbit of
  ${\it \Delta''}$. We draw an edge between two nodes whenever the
  polyhedral domains share a facet, here a face of dimension $9$.}
\label{fig:secondary-cone-tessellation}
\end{figure}

\subsection{Classification of Delaunay subdivisions}
\label{sec:classification-subdivisions}

There are $52$ different Delaunay subdivisions of quaternary positive
definite quadratic forms. This classification can be derived from
Voronoi's result. To accomplish this, one considers all faces of the
three simplicial cones ${\it \Delta}$, ${\it \Delta'}$,
${\it \Delta''}$ which gives $\leq 3072$ candidates, represented by
$\sum_{i=1}^{12} \alpha_i R_i$, with $\alpha_i \in \{0,1\}$. Then, one
checks which of these faces contain positive definite quadratic forms
in their relative interiors and which ones are equivalent under the
action of $\GL_4(\Z)$.

In Table~\ref{table:zonotopal-cases} and
Table~\ref{table:nonzonotopal-cases}, we give the classification
symbols of Conway. For every representative positive definite
quadratic form, we require the strict Voronoi vectors, which we list
in Table~\ref{table:strict-Voronoi-vectors}. For more details on this
classification and an explanation of the symbols used we refer to \cite{Conway-1997} and
\cite{Vallentin-2003}.

\begin{table}[htb]
\[
  \begin{split}
&    v_1 =  \pm (0, 0, 0, 1), \;
    v_2 =  \pm (0, 0, 1, -1), \;
    v_3 =  \pm (0, 0, 1, 0), \;
    v_4 =  \pm  (0, 0, 1, 1),\\
&    v_5 =  \pm (0, 1, 0, 0), \;
    v_6 = \pm ( 0, 1, 0, 1 ), \;
    v_7 = \pm ( 0, 1, 1, 0 ), \;
    v_8 = \pm (  0, 1, 1, 1 ),\\
&    v_9 = \pm ( 1, -1, 0, 0 ), \;
    v_{10} = \pm ( 1, 0, 0, 0 ), \;
    v_{11} = \pm ( 1, 0, 0, 1 ), \;
    v_{12} = \pm ( 1, 0, 1, 0 ),\\
&    v_{13} = \pm ( 1, 0, 1, 1 ), \;
    v_{14} = \pm ( 1, 1, 0, 0 ), \;
    v_{15} = \pm ( 1, 1, 0, 1 ), \;
    v_{16} = \pm ( 1, 1, 1, 0 ),\\
&    v_{17} = \pm ( 1, 1, 1, 1 )
    \end{split}
  \]
\caption{The complete list of all possible strict Voronoi vectors of
  positive definite quadratic forms in ${\it \Delta}$, ${\it
    \Delta'}$, ${\it \Delta''}$.}
\label{table:strict-Voronoi-vectors}
\end{table}

\begin{table}[htb]
  \begin{tabular}{lll}
  \hline
  lattice & dimension & representative\\
          & number & strict Voronoi vectors\\
    \hline
$K_5$ & $10$ & $R_1 + R_2 + R_3 + R_4 + R_5 + R_6 + R_7 + R_8 + R_9 +
        R_{10}$\\
           & $30$ & $v_{1}, v_{3}, v_{4}, v_{5}, v_{6}, v_{7}, v_{8},
               v_{10}, v_{11}, v_{12}, v_{13}, v_{14}, v_{15}, v_{16},
               v_{17}$ \\
    \hline
$ K_{3,3}$ & $9$ & $R_1 + R_2 + R_3 + R_4 + R_6 + R_7 + R_8 + R_9 +  R_{12}$ \\
          & $30$ & $v_{1}, v_{2}, v_{3}, v_{5}, v_{6}, v_{7}, v_{8}, v_{9}, v_{10}, v_{11}, v_{12}, v_{13}, v_{15}, v_{16}, v_{17}$\\
$K_5 - 1$ & $9$& $R_1 + R_2 + R_3 + R_4 + R_5 + R_7 + R_8 + R_9 + R_{10}$\\
          & $28$ & $v_{1}, v_{3}, v_{4}, v_{5}, v_{6}, v_{7}, v_{8}, v_{10}, v_{11}, v_{13}, v_{14}, v_{15}, v_{16}, v_{17}$
\\
    \hline
$K_5-2$  & $8$ & $R_1 + R_2 + R_3 + R_4 + R_7 + R_8 + R_9 + R_{10}$\\
         & $24$ & $v_{1}, v_{3}, v_{4}, v_{5}, v_{6}, v_{7}, v_{8}, v_{10}, v_{11}, v_{13}, v_{15}, v_{17}$\\
$K_5-1-1$ & $8$ & $R_1 + R_2 + R_3 + R_4 + R_5 + R_7 + R_8 + R_{10}$\\
         &  $26$ & $v_{1}, v_{3}, v_{4}, v_{5}, v_{7}, v_{8}, v_{10}, v_{11}, v_{13}, v_{14}, v_{15}, v_{16}, v_{17}$
\\
    \hline
    $K_5-3$  & $7$ & $R_1 + R_2 + R_4 + R_7 + R_8 + R_9 + R_{10}$\\
         & $20$ & $v_{1}, v_{3}, v_{4}, v_{5}, v_{7}, v_{8}, v_{10}, v_{11}, v_{13}, v_{17}$\\
    $K_5-2-1$ & $7$ & $R_1 + R_2 + R_4 + R_5 + R_7 + R_8 + R_{10}$\\
         & $24$ & $v_{1}, v_{3}, v_{4}, v_{5}, v_{7}, v_{8}, v_{10}, v_{11}, v_{13}, v_{14}, v_{16}, v_{17}$\\
    $K_4+1$  & $7$ & $R_1 + R_2 + R_3 + R_4 + R_8 + R_9 + R_{10}$\\
         & $16$ & $v_{1}, v_{3}, v_{4}, v_{5}, v_{6}, v_{7}, v_{8}, v_{10}$\\
    $C_{2221}$ & $7$ & $R_1 + R_2 + R_3 + R_4 + R_7 + R_9 + R_{10}$\\
         & $22$ & $v_{1}, v_{3}, v_{4}, v_{5}, v_{6}, v_{8}, v_{10}, v_{11}, v_{13}, v_{15}, v_{17}$
\\
\hline
    $C_{221} + 1$ &  $6$ & $R_1 + R_2 + R_3 + R_4 + R_8 + R_{10}$\\
         & $14$ & $v_{1}, v_{3}, v_{4}, v_{5}, v_{7}, v_{8}, v_{10}$\\
    $C_{321}$  & $6$ & $R_1 + R_2 + R_4 + R_7 + R_8 + R_{10}$\\
         & $20$ & $v_{1}, v_{3}, v_{4}, v_{5}, v_{7}, v_{8}, v_{10}, v_{11}, v_{13}, v_{17}$\\    
    $C_{222}$ & $6$ & $R_1+ R_2 + R_3 + R_7 + R_9 + R_{10}$\\
         & $22$ & $v_{1}, v_{3}, v_{4}, v_{5}, v_{6}, v_{8}, v_{10}, v_{11}, v_{13}, v_{15}, v_{17}$\\
    $C_3+C_3$  & $6$ & $R_1 + R_4 + R_7 + R_8 + R_9 + R_{10}$\\
         & $12$ & $v_{3}, v_{5}, v_{7}, v_{8}, v_{10}, v_{17}$\\
\hline
    $C_5$  & $5$ & $R_1 + R_2 + R_7 + R_8 + R_{10}$\\
         & $20$ & $v_{1}, v_{3}, v_{4}, v_{5}, v_{7}, v_{8}, v_{10}, v_{11}, v_{13}, v_{17}$\\
    $C_4+1$ & $5$ & $R_1 + R_2 + R_4 + R_8 + R_{10}$\\
         & $14$ & $v_{1}, v_{3}, v_{4}, v_{5}, v_{7}, v_{8}, v_{10}$\\
    $C_3+1+1$  & $5$ & $R_1 + R_2 + R_3 + R_4 + R_8$\\
        & $10$ & $v_{1}, v_{3}, v_{5}, v_{7}, v_{10}$\\
\hline
    $1+1+1+1$ & $4$ & $R_1 + R_2 + R_3 + R_4$\\
         & $8$ & $v_{1}, v_{3}, v_{5}, v_{10}$\\
\hline
  \end{tabular}

\smallskip

\caption{Representatives of the $17$ different four-dimensional
  Delaunay subdivisions where the corresponding Voronoi cell is a
  zonotope, that is, a linear projection of a cube. In Conway's
  classification the graph $K_4$ occurs, which is wrong and should be
  replaced by $C_{221} + 1$; this mistake was found by Deza and
  Grishukhin in \cite{Deza-2008}. The entry ``dimension'' gives the
  dimension of the secondary cone of the corresponding Delaunay
  subdivision. The entry ``number'' gives the number of strict Voronoi
  vectors.}

\label{table:zonotopal-cases}
\end{table}

\begin{table}[htb]
  \begin{tabular}{lll}
  \hline
  lattice & dimension & representative\\
          & number & strict Voronoi vectors\\ \hline   
  $111+$ & $10$ & $R_1 + R_2 + R_3 + R_4 + R_6 + R_7 + R_8 + R_9 + R_{11} + R_{12}$\\
  & $30$ & $v_{1}, v_{2}, v_{3}, v_{5}, v_{6}, v_{7}, v_{8}, v_{9}, v_{10},
    v_{11}, v_{12}, v_{13}, v_{15}, v_{16}, v_{17}$ \\
 $111-$ & $10$ & $R_1 + R_2 + R_3 + R_4 + R_6  + R_7 + R_8 + R_9 + R_{10} +
          R_{11}$\\
  & $30$ & $v_{1}, v_{3}, v_{4}, v_{5}, v_{6}, v_{7}, v_{8}, v_{9}, v_{10}, v_{11}, v_{12}, v_{13}, v_{15}, v_{16}, v_{17}$\\
  \hline
  $211+$ & $9$ & $R_1 + R_2 + R_3 + R_4 + R_6 + R_8 + R_9 + R_{11} +
           R_{12}$ \\
  & $28$ & $v_{1}, v_{2}, v_{3}, v_{5}, v_{6}, v_{7}, v_{8}, v_{9}, v_{10}, v_{11}, v_{12}, v_{13}, v_{16}, v_{17}$\\
  $211-$ & $9$ & $R_1 + R_2 + R_3 + R_4 + R_6  +  R_8 + R_9 + R_{10} +
           R_{11}$ \\
  & $28$ & $v_{1}, v_{3}, v_{4}, v_{5}, v_{6}, v_{7}, v_{8}, v_{9}, v_{10}, v_{11}, v_{12}, v_{13}, v_{16}, v_{17}$
\\
  \hline
  $311+$ & $8$ & $R_1 + R_2 + R_3 + R_4 + R_6 + R_8 + R_{11} + R_{12}$\\
  & $28$ & $v_{1}, v_{2}, v_{3}, v_{5}, v_{6}, v_{7}, v_{8}, v_{9}, v_{10}, v_{11}, v_{12}, v_{13}, v_{16}, v_{17}$
\\
  $311-$ & $8$ & $R_1 + R_2 + R_3 + R_4 + R_6  +  R_8 + R_{10} + R_{11}$\\
  & $28$ & $v_{1}, v_{3}, v_{4}, v_{5}, v_{6}, v_{7}, v_{8}, v_{9}, v_{10}, v_{11}, v_{12}, v_{13}, v_{16}, v_{17}$\\
  $221+$ & $8$ & $R_1 + R_2 + R_3 + R_4 + R_8 + R_9 + R_{11} + R_{12}$\\
    & $26$ & $v_{1}, v_{2}, v_{3}, v_{5}, v_{6}, v_{7}, v_{8}, v_{9}, v_{10}, v_{11}, v_{12}, v_{13}, v_{17}$\\
  $221-$ & $8$ & $R_1 + R_2 + R_3 + R_4 + R_8 + R_9 + R_{10} + R_{11}$\\
    & $26$ & $v_{1}, v_{3}, v_{4}, v_{5}, v_{6}, v_{7}, v_{8}, v_{9}, v_{10}, v_{11}, v_{12}, v_{13}, v_{17}$\\
  $22'1$ & $8$ & $R_1 + R_2 + R_3 + R_4 + R_7 + R_8  + R_{11} + R_{12}$\\
    & $26$ & $v_{1}, v_{2}, v_{3}, v_{5}, v_{6}, v_{7}, v_{8}, v_{9}, v_{10}, v_{11}, v_{12}, v_{13}, v_{17},
$\\
  \hline
  $411$ &  $7$ & $R_1+ R_2 + R_4 + R_6 + R_8 + R_{10} + R_{11}$\\
& $28$ & $v_{1}, v_{3}, v_{4}, v_{5}, v_{6}, v_{7}, v_{8}, v_{9}, v_{10}, v_{11}, v_{12}, v_{13}, v_{16}, v_{17}$\\
  $321+$ & $7$ & $R_1 + R_2 + R_4 + R_7 + R_8 + R_{10} + R_{11}$\\
  & $26$ & $v_{1}, v_{3}, v_{4}, v_{5}, v_{6}, v_{7}, v_{8}, v_{9}, v_{10}, v_{11}, v_{12}, v_{13}, v_{17}$\\
  $321-$ & $7$ & $R_1 + R_2 + R_3 + R_4 + R_8 + R_{10} + R_{11}$\\
  & $26$ & $v_{1}, v_{3}, v_{4}, v_{5}, v_{6}, v_{7}, v_{8}, v_{9}, v_{10}, v_{11}, v_{12}, v_{13}, v_{17}$\\
  $222+$ & $7$ & $R_1 + R_3 + R_4 + R_8 + R_9 + R_{11} + R_{12}$\\
  & $24$ & $v_{1}, v_{3}, v_{5}, v_{6}, v_{7}, v_{8}, v_{9}, v_{10}, v_{11}, v_{12}, v_{13}, v_{17}$\\
  $222-$ & $7$ & $R_1 + R_3 + R_4 + R_6  + R_7 + R_{10} + R_{11}$\\
  & $24$ & $v_{1}, v_{3}, v_{5}, v_{6}, v_{7}, v_{8}, v_{9}, v_{10}, v_{11}, v_{12}, v_{13}, v_{17}$\\
  $222'$ & $7$ & $R_1 + R_3 + R_4 + R_8 + R_9 + R_{10} + R_{11}$\\
  & $24$ & $v_{1}, v_{3}, v_{5}, v_{6}, v_{7}, v_{8}, v_{9}, v_{10}, v_{11}, v_{12}, v_{13}, v_{17}$\\
    $22'2''$ & $7$  & $R_1 + R_4 + R_7 + R_8 + R_9 + R_{10} + R_{11}$\\
    & $24$ & $v_{1}, v_{3}, v_{5}, v_{6}, v_{7}, v_{8}, v_{9}, v_{10}, v_{11}, v_{12}, v_{13}, v_{17}$\\
  \hline
\end{tabular}

\smallskip

\caption{(First part) Representatives of the $35$ different four-dimensional Delaunay
  subdivisions where the corresponding Voronoi cell is not a zonotope.}

\label{table:nonzonotopal-cases}

\end{table}

\addtocounter{table}{-1}

 \begin{table}[htb]
  \begin{tabular}{lll}
  \hline
  lattice & dimension & representative\\
          & number & strict Voronoi vectors\\ \hline   
  $421$  & $6$ & $R_1 + R_2 + R_4 + R_8 + R_{10} + R_{11}$\\
    & $26$ & $v_{1}, v_{3}, v_{4}, v_{5}, v_{6}, v_{7}, v_{8}, v_{9}, v_{10}, v_{11}, v_{12}, v_{13}, v_{17}$\\
  $331+$ & $6$ & $R_1 + R_2 + R_3 + R_4 + R_{11} + R_{12}$\\
    & $26$ & $v_{1}, v_{2}, v_{3}, v_{5}, v_{6}, v_{7}, v_{8}, v_{9}, v_{10}, v_{11}, v_{12}, v_{13}, v_{17}$\\
  $331-$ & $6$ & $R_1+ R_2 + R_3 + R_4 + R_{10} + R_{11}$\\
    & $26$ & $v_{1}, v_{3}, v_{4}, v_{5}, v_{6}, v_{7}, v_{8}, v_{9}, v_{10}, v_{11}, v_{12}, v_{13}, v_{17},
$\\
  $322+$ & $6$ & $R_1 + R_3 + R_4 + R_8 + R_{11} + R_{12}$\\
    &  $24$ & $v_{1}, v_{3}, v_{5}, v_{6}, v_{7}, v_{8}, v_{9}, v_{10}, v_{11}, v_{12}, v_{13}, v_{17}$\\
  $322-$ & $6$ & $R_1 + R_4 + R_7 + R_8 + R_{11} + R_{12}$\\
    & $24$ & $v_{1}, v_{3}, v_{5}, v_{6}, v_{7}, v_{8}, v_{9}, v_{10}, v_{11}, v_{12}, v_{13}, v_{17}$\\
  $322'$ & $6$ & $R_1 + R_2 + R_4 + R_6 + R_7 + R_{11}$\\
  & $24$ & $v_{1}, v_{3}, v_{5}, v_{6}, v_{7}, v_{8}, v_{9}, v_{10}, v_{11}, v_{12}, v_{13}, v_{17}$\\
\hline
  $431$ & $5$ & $R_1 + R_2 + R_4 + R_{10} + R_{11}$\\
    & $26$ & $v_{1}, v_{3}, v_{4}, v_{5}, v_{6}, v_{7}, v_{8}, v_{9}, v_{10}, v_{11}, v_{12}, v_{13}, v_{17}$\\
  $422$ & $5$ & $R_1 + R_4 + R_8 + R_{11} + R_{12}$\\
    & $24$ & $v_{1}, v_{3}, v_{5}, v_{6}, v_{7}, v_{8}, v_{9}, v_{10}, v_{11}, v_{12}, v_{13}, v_{17}$\\
  $422'$ & $5$ & $R_1 + R_4 + R_8 + R_{10} + R_{11}$\\
    & $24$ & $v_{1}, v_{3}, v_{5}, v_{6}, v_{7}, v_{8}, v_{9}, v_{10}, v_{11}, v_{12}, v_{13}, v_{17}$\\
  $332+$ & $5$ & $R_1 + R_3 + R_4 + R_{11} + R_{12}$\\
    & $24$ & $v_{1}, v_{3}, v_{5}, v_{6}, v_{7}, v_{8}, v_{9}, v_{10}, v_{11}, v_{12}, v_{13}, v_{17}$\\
  $332-$ & $5$ & $R_1 + R_3 + R_4 +R_{10} + R_{11}$\\
    & $24$ & $v_{1}, v_{3}, v_{5}, v_{6}, v_{7}, v_{8}, v_{9}, v_{10}, v_{11}, v_{12}, v_{13}, v_{17}$\\
  \hline
  $441$ & $5$ &  $R_1 + R_2 + R_{10} + R_{11}$ \\
      & $26$ & $v_{1}, v_{3}, v_{4}, v_{5}, v_{6}, v_{7}, v_{8}, v_{9}, v_{10}, v_{11}, v_{12}, v_{13}, v_{17}$\\
  $432$ & $4$ & $R_1 + R_4 + R_{10} + R_{11}$\\
      & $24$ & $v_{1}, v_{3}, v_{5}, v_{6}, v_{7}, v_{8}, v_{9}, v_{10}, v_{11}, v_{12}, v_{13}, v_{17}$\\
  $333+$ & $4$ & $R_3 + R_4 + R_{11} + R_{12}$\\
      & $24$ & $v_{1}, v_{3}, v_{5}, v_{6}, v_{7}, v_{8}, v_{9}, v_{10}, v_{11}, v_{12}, v_{13}, v_{17}$\\
  $333-$ & $4$ & $R_3 + R_4 + R_{10} + R_{11}$\\
   & $24$ & $v_{1}, v_{3}, v_{5}, v_{6}, v_{7}, v_{8}, v_{9}, v_{10}, v_{11}, v_{12}, v_{13}, v_{17}$\\
  \hline
  $442$  &  $3$ & $R_1 + R_{10} + R_{11}$\\
     & $24$ & $v_{1}, v_{3}, v_{5}, v_{6}, v_{7}, v_{8}, v_{9}, v_{10}, v_{11}, v_{12}, v_{13}, v_{17}$\\
  $433$ &  $3$ & $R_4 + R_{10} + R_{11}$\\
     & $24$ & $v_{1}, v_{3}, v_{5}, v_{6}, v_{7}, v_{8}, v_{9}, v_{10}, v_{11}, v_{12}, v_{13}, v_{17}$\\
  \hline
  $443$ & $2$ & $R_{10} + R_{11}$\\
     & $24$ & $v_{1}, v_{3}, v_{5}, v_{6}, v_{7}, v_{8}, v_{9}, v_{10}, v_{11}, v_{12}, v_{13}, v_{17}$
\\
  \hline
  $444$ & $1$ & $R_{11}$\\
     & $24$ & $v_{1}, v_{3}, v_{5}, v_{6}, v_{7}, v_{8}, v_{9}, v_{10}, v_{11}, v_{12}, v_{13}, v_{17}$\\
  \hline
  \end{tabular}

  \smallskip
  
\caption{(Second part) Representatives of the $35$ different four-dimensional Delaunay
  subdivisions where the corresponding Voronoi cell is not a
  zonotope.}

\end{table}

\subsection{Classification of Voronoi graphs}
\label{sec:classification-graphs}

In this section, we derive a classification of Voronoi graphs for
four-dimensional lattices based on the classification of Delaunay
subdivisions. We have to verify which of the above $52$ Voronoi graphs
are isomorphic.

\smallskip

In many cases, the strict Voronoi vectors coincide, so that also the
Voronoi graphs coincide.  In three cases, for example, for the pairs
$(411, 311+)$, $(441, 331+)$, $(K_5-2, K_5 - 2 -1)$ the strict Voronoi
vectors do not coincide, but the Voronoi graphs are still
isomorphic. In the first two cases one can find an isomorphism by
determining an element $A \in \GL_4(\Z)$ of the automorphism group of
$444$ (which corresponds to the root lattice $D_4$ and which has order
$1152$) so that the Voronoi vectors of $411$ are mapped to the ones of
$311+$, respectively of $441$ to $331+$. This linear isomorphism
between the generators of the Cayley graph induces a global graph
isomorphism. In the third case, one can use an element of the
automorphism group of $K_5$ (which corresponds to the weight lattice
$A^*_4$ and which has order $240$).

\smallskip

Then we are left with $16$ candidates of pairwise non-isomorphic
Voronoi graphs. They are indeed pairwise non-isomorphic, which one can
see by looking at invariants. The first invariant we are using is the
regularity $r$ of the graphs. The second and third invariant we are
using are coming from the finite, induced subgraph
$\Cayley(\Lambda, \Vor(\Lambda))_1$ whose vertices are at most unit
distance from the origin. The second invariant is the number of edges
$|E|$ of $\Cayley(\Lambda, \Vor(\Lambda))_1$ and the third invariant
is the order of the (graph) automorphism group of
$\Cayley(\Lambda, \Vor(\Lambda))_1$.

\begin{table}[htb]
\begin{tabular}{lllllll}
  \hline
 \# & graph & $r$ & $|E|$ & order & number & representatives\\
  \hline
  1 & $V_{30}^z$ & $30$ & $180$ & $240$ & $1$ & $K_5$\\
  2 &  $V_{30}^{z,n}$ & $30$ & $186$ & $144$ & $2$ & $K_{3,3}$, $111+$\\
  3 &  $V_{30}^n$ & $30$ & $180$ & $24$ & $1$ & $111-$ \\
  \hline
  4 &  $V_{28}^z$ & $28$ & $154$ & $24$ & $1$ & $K_5 - 1$\\
  5 &  $V_{28}^n$ & $28$ & $160$ & $24$ & $5$ & $211+$, $211-$,
                                                $311+$, $311-$,
                                                $411$\\
    \hline
  6 &  $V_{26}^z$ & $26$ & $134$ & $16$ & $1$ & $K_5 - 1 - 1$\\
  7 &  $V_{26}^n$ & $26$ & $140$ & $96$ & $10$ & $221+$, $221-$, $22'1$, $321+$,
                                 $321-$,\\
         & & & &  & & $421$, $331+$, $331-$, $431$,  $441$\\
    \hline
  8 &  $V_{24}^z$ & $24$ & $114$ & $16$ & $2$ & $K_5-2$, $K_5-2-1$ \\
  9 &  $V_{24}^n$ & $24$ & $120$ & $1152$ & $18$ & $222+$, $222-$, $222'$, $22'2''$,
                                  $322+$,\\
         & & & &  & & $322-$, $322'$, $422$, $422'$, $332+$,\\
         & & & & & & $332-$, $432$, $333+$, $333-$, $442$, \\
         & & & & & & $433$, $443$, $444$ \\
    \hline
  10 &  $V_{22}^z$ & $22$ & $94$ & $96$ & $2$ & $C_{2221}$, $C_{222}$\\
  11 &  $V_{20}^z$ & $20$ & $80$ & $240$ & $3$ & $K_5 - 3$, $C_{321}$, $C_5$ \\
  12 &  $V_{16}^z$ & $16$ & $52$ & $96$ & $1$ & $K_4+1$ \\
  13 &  $V_{14}^z$ & $14$ & $38$ & $96$ & $2$ & $C_{221} + 1$, $C_4 + 1$ \\
  14 &  $V_{12}^z$ & $12$ & $24$ & $288$ & $1$ & $C_3+C_3$ \\
  15 &  $V_{10}^z$ & $10$ & $16$ & $288$ & $1$ & $C_3+1+1$\\
  16 &  $V_{8}^z$ & $8$ & $8$ & $40320$ & $1$ & $1+1+1+1$\\
  \hline
\end{tabular}
  
\smallskip

\caption{Classification of Voronoi graphs of four-dimensional
  lattices. The column $r$ gives the regularity of the Voronoi graph,
  the column $|E|$ gives the number of edges, the column ``order'' gives the
  order of the automorphism group of the induced subgraph
  $\Cayley(\Lambda, \Vor(\Lambda))_1$.}

\label{table:Cayley-graphs}

\end{table}

Here we were able to test whether two Voronoi graphs are isomorphic in a rather ad hoc way. 
This can also be done in a systematic way: 
In the next section we provide a finite algorithm that can be used in general, for this we will briefly discuss how to test
for isomorphism of general Cayley graphs associated with lattices with
finite sets of generators; see also the paper \cite{Baburin-2020} by Baburin 
which summarizes computational approaches to finding isomorphisms and automorphisms of Cayley graphs of lattices and contains references to earlier works in this field.

\subsection{Isomorphisms of Cayley graphs of lattices}

Suppose $\Lambda$ and $\Lambda'$ are lattices, and
$S \subseteq \Lambda$ and $S' \subseteq \Lambda'$ are finite, centrally
symmetric generating sets. That is, $S = -S$ and
$\Lambda = \langle S \rangle_{\mathbb{Z}}$, and the corresponding
identities hold for $\Lambda'$ and $S'$.

Any linear isomorphism between $\Lambda$ and $\Lambda'$ that maps $S$
to $S'$ induces a graph isomorphism between $\Cayley(\Lambda,S)$ and
$\Cayley(\Lambda',S')$.  A general result on isomorphisms of Cayley
graphs of finitely generated abelian groups (see L\"oh
\cite{Loeh-2013}) implies that the converse is also true: Any graph
isomorphism
$\varphi : \Cayley(\Lambda,S) \rightarrow \Cayley(\Lambda',S')$
induces a linear isomorphism of $\Lambda$ and $\Lambda'$ which sends $S$ to $S'$.

Since we think this result is of independent interest in the context
of lattices, and the elegant proof of the general result can be
significantly shortened in this context, we reproduce a shortened
version of the proof given in L\"oh \cite{Loeh-2013} here.

The main idea of the proof is to use induction on the size of the
generating set $S$, the crucial observation is that this induction
should be done by reducing the size of $S$ by removing an
$\|\cdot\|_2$-maximal element from $S$.

So, suppose $\Lambda$ is a lattice with a centrally symmetric generating
set $S$. Given $s \in S$ we write $S_s = S \setminus \{s,-s\}$ and
$\Lambda_s = \langle S_s \rangle_{\mathbb{Z}}$.

We rest the induction step on the following lemma.

\begin{lemma}
  \label{lem:cayley:reduction}
  Let $s \in S$ be $\|\cdot\|_2$-maximal and let
  $\varphi : \Cayley(\Lambda,S) \rightarrow \Cayley(\Lambda',S')$ be a graph
  isomorphism. Then the restriction $\varphi_s$ of $\varphi$ to
  $\Cayley(\Lambda_s,S_s)$ induces a graph isomorphism
  $\Cayley(\Lambda_s,S_s) \rightarrow \Cayley(\Lambda'_{s'},S'_{s'})$, where
  $s' = \varphi(s)$.
\end{lemma}

\begin{proof}
  The result will be an immediate consequence of Lemma
  \ref{lem:algebraic:convexgeodesic:correspondence}
\end{proof}

The workhorse of the proof is the following combination of Proposition
2.5 and Proposition 2.10 in \cite{Loeh-2013}. For this we need to
introduce some concepts from the theory of (Cayley) graphs.  Firstly,
an \emph{algebraic line} of type $s \in S$ in $\Cayley(\Lambda,S)$ is
a $\Z$-path $\gamma: \Z \rightarrow \Lambda$ with $n \mapsto v + ns$, for some
$v\in \Lambda$.  Secondly, a \emph{geodesic segment} in a graph
$\Gamma=(V,E)$ is a finite path $(v_0,\ldots,v_n)$ with shortest path
distance $d_{\Gamma}(v_0,v_n) = n$ and a \emph{geodesic line} is a
$\Z$-path $\gamma:\Z \rightarrow V$, such that
$d_{\Gamma}(\gamma(i),\gamma(j)) = |i-j|$ for all $i,j \in \Z$.
Lastly, a \emph{convex geodesic line} in $\Gamma$ is a geodesic line
$\gamma$, such that for all $m,n \in \Z$ with $m\leq n$ there is a
unique geodesic segment in $\Gamma$ starting at $\gamma(m)$ and ending
at $\gamma(n)$.

\begin{lemma}
  \label{lem:algebraic:convexgeodesic:correspondence}
  Let $\Lambda,\Lambda'$ be lattices and $S,S'$ be finite, centrally symmetric
  generating sets of $\Lambda,\Lambda'$ and let
  $\varphi : \Cayley(\Lambda,S) \rightarrow \Cayley(\Lambda',S')$ be a graph
  isomorphism. Then:
  \begin{enumerate}
  \item If $s$ is a $\|\cdot\|_2$-maximal element of $S$, then all
    algebraic lines of type $s$ are convex geodesic lines.
    \item $\varphi$ maps convex geodesic lines to convex geodesic lines.
    \item Every convex geodesic line is algebraic.
    \item If $\gamma$ and $\eta$ are algebraic lines of type $s$ in
      $\Cayley(\Lambda,S)$, then $\varphi \circ \gamma$ and
      $\varphi \circ \eta$ are algebraic lines of the same type in
      $\Cayley(\Lambda',S')$. In particular, their type is
      $s' = \varphi(s)$.
  \end{enumerate}
\end{lemma}

\begin{proof}
  (1) through (3) make up Proposition 2.5 in \cite{Loeh-2013}.

  We give a short sketch of the proof of (1) to highlight the
  importance of $\|\cdot\|_2$-maximality of $s$.

  For this choose an arbitrary geodesic segment $(g_0,\ldots,g_k)$ in
  $\Cayley(\Lambda,S)$ connecting points $\gamma(m),\gamma(n)$ on $\gamma$
  (assume $m \leq n$). Write
  \[
    (n-m) s = \gamma(n)-\gamma(m) = \sum_{i=0}^{k-1} g_{i+1} - g_i  
  \]
  where $g_{i+1} - g_i \in S$. Then by the maximality of $s$ we can
  conclude, using the triangle inequality, that $n-m = k$,
  $\|g_{i+1}-g_i\|_2 = \|s\|_2$ and then $g_{i+1}-g_i = s$. From this
  it follows that the arbitrary geodesic segment $(g_0,\ldots,g_k)$ is
  a segment on $\gamma$ as needed.

  (4) is a simplification of Proposition 2.10 in \cite{Loeh-2013}. We
  present a streamlined version of the proof here:
  Write $\gamma' = \varphi \circ \gamma$ and $\eta' = \varphi \circ \eta$ and
  elementwise for $n \in \Z$
  \begin{align*}
    \gamma(n) &= v + n s\\
    \eta(n) &= w + n s\\
    \gamma'(n) &= v'+ n s'\\
    \eta'(n) &= w' + n t'.
  \end{align*} 
  We have to show $s'= t'$. We consider the graph distances $d$ and
  $d'$ on $\Cayley(\Lambda,S)$ and $\Cayley(\Lambda',S')$. Since
  $\varphi$ is an isometry and the graph distance is translation
  invariant we can compute
  \begin{align*}
    d'(n(s'-t'),w'-v') &= d'(\gamma'(n),\eta'(n)) \\
    &= d(\gamma(n),\eta(n))\\
    &= d(n(s-s),w-v) = d(0,w-v)
  \end{align*}
  and observe that the right hand side is independent of $n$. Therefore,
  the set $\{ n (s'-t') \ :\ n \in \Z \}$ is contained in a ball of
  finite radius around $w'-v'$ and is hence finite. This implies
  that $s'-t'$ has finite order in $\Lambda$ and since $\Lambda$ is torsion free,
  $s'-t' = 0$.
\end{proof}

Now we can prove the isomorphism characterization.

\begin{theorem}
  \label{thm:cayley:isomorphism}
  Let $\Lambda,\Lambda'$ be lattices and $S,S'$ be finite centrally
  symmetric generating sets of $\Lambda,\Lambda'$. If
  $\varphi : \Cayley(\Lambda,S) \rightarrow \Cayley(\Lambda',S')$ is
  a graph isomorphism, then $\varphi$ is an affine map. In particular,
  if $\varphi(0) = 0$, then $\varphi$ is linear.
\end{theorem}

\begin{proof}
  We can assume that $\varphi(0) = 0$, since this can always be
  achieved by composing $\varphi$ with a suitable translation.

  The proof then proceeds by induction on the size of $S$.  Let
  $s \in S$ be a $\|\cdot\|_2$-maximal element of $S$.  By Lemma
  \ref{lem:cayley:reduction} we know that the restriction $\varphi_s$
  induces a graph isomorphism
  $\Cayley(\Lambda_s,S_s) \cong \Cayley(\Lambda'_{s'},S'_{s'})$.  By induction
  hypothesis, this is a linear isomorphism.

  Now we show that this implies that also $\varphi$ is linear. For
  this take arbitrary $v,w \in \Lambda$ and write them as
  \[
    v = x + m s, \quad w = y + n s
  \]
  with $x,y \in \Lambda_s$, $m,n \in \Z$.
  Then
  \[
    \varphi(v) = \varphi(x) + m s', \ \varphi(w) = \varphi(y) + n s', \ \text{and} \ \varphi(v+w) = \varphi(x+y) + (m+n)s'.
  \]
  Since $\varphi$ acts on $\Lambda_s$ as $\varphi_s$, and the latter
  is additive, we obtain
  \[
    \varphi(v+w) = \varphi(x+y) + (m+n)s' = (\varphi(x) + m s') + (\varphi(y) + n s') = \varphi(v) + \varphi(w).
  \]
  So $\varphi$ is indeed additive. 
\end{proof}

This approach suggests a finite algorithm to test isomorphism of
Cayley graphs of two lattices: Let $\Lambda$ and $\Lambda'$ be
lattices with finite centrally symmetric generating sets $S$ and
$S'$. For simplicity, assume that both $\Lambda$ and $\Lambda'$ are
full-dimensional in their respective dimensions.

\begin{enumerate}
 \item First, fix an arbitrary (vector space) basis contained in $S$, say
   $s_1,\ldots,s_n$.
   \item Second, consider all linear maps defined by sending $s_1,\ldots,s_n$ to $s'_1,\ldots,s'_n$, where the latter runs through
all bases contained in $S'$.
\item For each such map $\varphi$ we now check whether $S$ is mapped
  onto $S'$.  If that is the case, $\varphi$ is an isomorphism of the
  associated Cayley graphs.
\end{enumerate}

This algorithm terminates, because $S'$ is finite. It is correct
because Theorem \ref{thm:cayley:isomorphism} asserts that any
isomorphism of the associated Cayley graphs is of this form.

\section{Bounds for $\chi(\Lambda)$ and proof of
  Theorem~\ref{thm:main}}
\label{sec:bounds}

Now we consider each of the $16$ cases identified in Section
\ref{sec:classification-graphs} separately.  Assume
$C = \Cayley(\Lambda, \Vor(\Lambda))$ is the Cayley graph of a lattice
$\Lambda$.

First, in Section \ref{sec:dpb}, we will consider the finite induced
subgraph $$C_1 = \Cayley(\Lambda, \Vor(\Lambda))_1$$ with vertices
$\{0\} \cup \Vor(\Lambda)$.  This subgraph captures the combinatorial
structure, the vertex-edge graph, of all Delaunay polytopes of
$\Lambda$ containing the origin as a vertex.  The chromatic number of
this subgraph provides a lower bound for $\chi(\Lambda)$ which we call the
\emph{Delaunay polytope bound (DPB)}.  We can explicitly compute this
bound by reformulating the chromatic number problem in
terms of a Boolean satisfiability problem (SAT) and using a SAT
solver.

In Section \ref{sec:dtb} we discuss a method to generate periodic
colorings of $C$, which we then use to provide an upper bound on
$\chi(\Lambda)$, which we call the \emph{discrete torus bound (DTB)}.
The validity of the proposed periodic colorings comes from a
quotient-like construction of finite graphs and it is again verified through a
reformulation as a SAT problem.

In Section \ref{ssec:proof} we explain the computational proof of
the main theorem. We recall how to reformulate the determination of
the chromatic number as a SAT formula. We also comment on the
accompanying Magma program used to perform the involved
computational verifications.

\subsection{The Delaunay polytope bound}
\label{sec:dpb}

Consider the induced subgraph $C_1$ of the Cayley graph
$C = \Cayley(\Lambda, \Vor(\Lambda))$ of a lattice $\Lambda$ with
vertex set $V_1 = \{0\} \cup \Vor(\Lambda)$.  This is precisely the
vertex-edge graph of the polyhedral complex of all Delaunay polytopes
of $\Lambda$ containing the origin.  Indeed if $[x,y]$ is an edge of
any such Delaunay polytope $D$, then $x-y$ is a relevant vector, since
the Voronoi cells $x + V(\Lambda)$ and $y + V(\Lambda)$ meet in a
common facet by duality of Voronoi and Delaunay tessellations. This
gives us the \emph{Delaunay polytope bound}.

\begin{lemma} \label{lem:dpb}
  With the above
  \[
    \chi(\Lambda) = \chi(C) \geq \chi(C_1).
 \]
\end{lemma}

\begin{figure}\label{fig:dpb}
   \tikzstyle{polygon}=[thin]
  \tikzstyle{point}=[ultra thin,draw=gray,cross out,inner sep=0pt,minimum width=4pt,minimum height=4pt,]
  \tikzstyle{line}=[red]
  \begin{tikzpicture}[scale=1.5]

      \coordinate (A) at (-.66,-.33);
      \coordinate (B) at (-.33,  .33);
      \coordinate (D) at (.66, .33);
      \coordinate (C) at (.33, .66);
      \coordinate (F) at (-.33, -.66);
      \coordinate (E) at (.33, -.33);

      \draw[polygon] (A) -- (B) -- (C) -- (D) -- (E) -- (F)  -- cycle;


      
      \coordinate (A1) at (.33,-.33);
      \coordinate (B1) at (.66,  .33);
      \coordinate (C1) at (1.33, .66);
      \coordinate (D1) at (1.66, .33);
      \coordinate (E1) at (1.33, -.33);
      \coordinate (F1) at (.66, -.66);
      
      \draw[polygon]  (A1) -- (B1) -- (C1) -- (D1) -- (E1) -- (F1)  -- cycle;
      \coordinate (A2) at (.33,.66);
      \coordinate (B2) at (.66,  1.33);
      \coordinate (C2) at (1.33, 1.66);
      \coordinate (D2) at (1.66, 1.33);
      \coordinate (E2) at (1.33, .66);
      \coordinate (F2) at (.66, .33);
  \draw[polygon]  (A2) -- (B2) -- (C2) -- (D2) -- (E2) -- (F2)  -- cycle;

  \coordinate (A3) at (-.66,.66);
      \coordinate (B3) at (-.33,  1.33);
      \coordinate (D3) at (.66, 1.33);
      \coordinate (C3) at (.33, 1.66);
      \coordinate (F3) at (-.33, .33);
      \coordinate (E3) at (.33, .66);
  \draw[polygon]  (A3) -- (B3) -- (C3) -- (D3) -- (E3) -- (F3)  -- cycle;

  \coordinate (A4) at (-.66,1.66);
      \coordinate (B4) at (-.33,  2.33);
      \coordinate (D4) at (.66, 2.33);
      \coordinate (C4) at (.33, 2.66);
      \coordinate (F4) at (-.33, 1.33);
      \coordinate (E4) at (.33, 1.66);

  \coordinate (A5) at (.33,1.66);
      \coordinate (B5) at (.66,  2.33);
      \coordinate (C5) at (1.33, 2.66);
      \coordinate (D5) at (1.66, 2.33);
      \coordinate (E5) at (1.33, 1.66);
      \coordinate (F5) at (.66, 1.33);
  \draw[polygon]  (A5) -- (B5) -- (C5) -- (D5) -- (E5) -- (F5)  -- cycle;

  \coordinate (A6) at (1.33,-.33);
      \coordinate (B6) at (1.66,  .33);
      \coordinate (C6) at (2.33, .66);
      \coordinate (D6) at (2.66, .33);
      \coordinate (E6) at (2.33, -.33);
      \coordinate (F6) at (1.66, -.66);
      

  \coordinate (A7) at (1.33,.66);
      \coordinate (B7) at (1.66,  1.33);
      \coordinate (C7) at (2.33, 1.66);
      \coordinate (D7) at (2.66, 1.33);
      \coordinate (E7) at (2.33, .66);
      \coordinate (F7) at (1.66, .33);
      
      \draw[polygon]  (A7) -- (B7) -- (C7) -- (D7) -- (E7) -- (F7)  -- cycle;

      \coordinate (A8) at (1.33,1.66);
      \coordinate (B8) at (1.66,  2.33);
      \coordinate (C8) at (2.33, 2.66);
      \coordinate (D8) at (2.66, 2.33);
      \coordinate (E8) at (2.33, 1.66);
      \coordinate (F8) at (1.66, 1.33);
      
      \draw[polygon]  (A8) -- (B8) -- (C8) -- (D8) -- (E8) -- (F8)  -- cycle;

      \coordinate (0) at (0,0);
      \coordinate (1) at (1,0);
      \coordinate (2) at (0, 1);
      \coordinate (3) at (1, 1);
      \coordinate (4) at (2,0);
      \coordinate (5) at (0,2);
      \coordinate (6) at (1,2);
      \coordinate (7) at (2,1);
      \coordinate (8) at (2,2);

      \draw[lcbeige] (0,0)--(1,0);
      \draw[lcbeige] (0,1)--(2,1);
      \draw[lcbeige] (1,2)--(2,2);

      \draw[lcbeige] (0,0)--(0,1);
      \draw[lcbeige] (1,0)--(1,2);
      \draw[lcbeige] (2,1)--(2,2);

      \draw[lcbeige] (0,1)--(1,2);
      \draw[lcbeige] (0,0)--(2,2);
      \draw[lcbeige] (1,0)--(2,1);




  \node[circle,inner sep=.6pt,fill=black,label={[scale=.5]below left:$(0,-1)$}] at (1,0) {};
  \node[circle,inner sep=.6pt,fill=black,label={[scale=.5]above left:$(-1,0)$}] at (0,1) {};
  \node[circle,inner sep=.6pt,fill=black,label={[scale=.5]below left:$(-1,-1)$}] at (0,0) {};
  \node[circle,inner sep=.6pt,fill=black,label={[scale=.5]above left:$(0,1)$}] at (1, 2) {};

  \draw[lcred] (1,1) -- (2,2);
  \draw[lcred] (2,1) -- (2,2);
  \draw[lcred] (1,1) -- (2,1);

  \draw (1.66,1.33) ellipse [rotate=135,x radius=.471, y radius=.813];
  \node[circle,inner sep=1.5pt,fill=lcblue,label={[scale=.5]above right:$(1,1)$}] at (2,2) {};
  \node[circle,inner sep=1.5pt,fill=lcblue,label={[scale=.5]below right:$(1,0)$}] at (2,1) {};
  \node[circle,inner sep=1.5pt,fill=lcblue,label={[scale=.5]above left:$(0,0)$}] at (1,1) {};
  \node[circle,inner sep=1.5pt,fill=lcblue,label={[scale=.5]left:$c$}] at (5/3,4/3) {};
  \end{tikzpicture}

  \caption{A Delaunay polytope (red) of the hexagonal lattice included in the vertex-edge graph of the polyhedral complex of all Delaunay polytopes of $\Lambda$ containing the origin and the
    Voronoi cells centered at its vertices. The red edges are a finite subgraph of the Voronoi graph of the hexagonal lattice and
    provide a certificate that $\chi(A_2) \geq 3$.}
\end{figure}
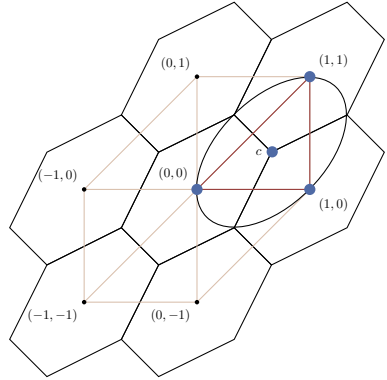

\subsection{The discrete torus bound}
\label{sec:dtb}

The goal of this subsection is to explicitly construct periodic colorings
of a lattice $\Lambda$. For this we construct a finite auxiliary
object, which we will refer to as a discrete torus of $\Lambda$: Let
$\Lambda' \subseteq \Lambda$ be a full dimensional sublattice, this
induces a graph, the \emph{discrete torus} of $\Lambda$ by $\Lambda'$,
with vertex set $V = \Lambda/\Lambda'$ and edge set
\[
  E = \{ \{x + \Lambda',y + \Lambda'\} \ :\ \text{there is } v \in
\Lambda' \text{ with }  v + x-y \in \Vor(\Lambda) \}. 
\]

\begin{lemma}
  \label{lem:dtb}
  Let $\Lambda' \subseteq \Lambda$ be a sublattice of $\Lambda$
  that does not contain any Voronoi relevant vector of $\Lambda$. Then
  a coloring of $\Lambda/\Lambda'$ can be $\Lambda'$-periodically
  extended to a coloring of $\Lambda$. If $\chi(\Lambda/\Lambda') < \infty$ we have $\chi(\Lambda) \leq \chi(\Lambda/\Lambda')$.
\end{lemma}

\begin{proof}
  By assumption $\Lambda'$ does not contain any relevant Voronoi
  vectors of $\Lambda$, this implies that every coset $v + \Lambda'$
  is an independent set in $\Cayley(\Lambda,\Vor(\Lambda))$. Thus any
  coloring of $\Lambda / \Lambda'$ can be extended
  $\Lambda'$-periodically as claimed and therefore
  $\chi(\Lambda/\Lambda') \geq \chi(\Lambda)$.
\end{proof}

For any choice of a full-dimensional $\Lambda'$ we refer to the bound of the above lemma
as the \emph{discrete torus bound} induced by $\Lambda'$.

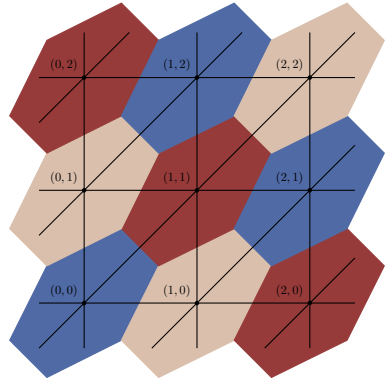
\begin{figure}
  \tikzstyle{polygon}=[thin]
  \tikzstyle{point}=[ultra thin,draw=gray,cross out,inner sep=0pt,minimum width=4pt,minimum height=4pt,]
  \tikzstyle{line}=[red]
  \begin{tikzpicture}[scale=1.5]
      \coordinate (A) at (-.66,-.33);
      \coordinate (B) at (-.33,  .33);
      \coordinate (D) at (.66, .33);
      \coordinate (C) at (.33, .66);
      \coordinate (F) at (-.33, -.66);
      \coordinate (E) at (.33, -.33);

      \draw[ fill,lcblue!100] (A) -- (B) -- (C) -- (D) -- (E) -- (F)  -- cycle;


      
      \coordinate (A1) at (.33,-.33);
      \coordinate (B1) at (.66,  .33);
      \coordinate (C1) at (1.33, .66);
      \coordinate (D1) at (1.66, .33);
      \coordinate (E1) at (1.33, -.33);
      \coordinate (F1) at (.66, -.66);
      
      \draw[fill,lcbeige!100]  (A1) -- (B1) -- (C1) -- (D1) -- (E1) -- (F1)  -- cycle;
      \coordinate (A2) at (.33,.66);
      \coordinate (B2) at (.66,  1.33);
      \coordinate (C2) at (1.33, 1.66);
      \coordinate (D2) at (1.66, 1.33);
      \coordinate (E2) at (1.33, .66);
      \coordinate (F2) at (.66, .33);
  \draw[fill,lcred!100]  (A2) -- (B2) -- (C2) -- (D2) -- (E2) -- (F2)  -- cycle;

  \coordinate (A3) at (-.66,.66);
      \coordinate (B3) at (-.33,  1.33);
      \coordinate (D3) at (.66, 1.33);
      \coordinate (C3) at (.33, 1.66);
      \coordinate (F3) at (-.33, .33);
      \coordinate (E3) at (.33, .66);
  \draw[fill,lcbeige!100]  (A3) -- (B3) -- (C3) -- (D3) -- (E3) -- (F3)  -- cycle;

  \coordinate (A4) at (-.66,1.66);
      \coordinate (B4) at (-.33,  2.33);
      \coordinate (D4) at (.66, 2.33);
      \coordinate (C4) at (.33, 2.66);
      \coordinate (F4) at (-.33, 1.33);
      \coordinate (E4) at (.33, 1.66);
  \draw[fill,lcred!100]  (A4) -- (B4) -- (C4) -- (D4) -- (E4) -- (F4)  -- cycle;

  \coordinate (A5) at (.33,1.66);
      \coordinate (B5) at (.66,  2.33);
      \coordinate (C5) at (1.33, 2.66);
      \coordinate (D5) at (1.66, 2.33);
      \coordinate (E5) at (1.33, 1.66);
      \coordinate (F5) at (.66, 1.33);
  \draw[fill,lcblue!100]  (A5) -- (B5) -- (C5) -- (D5) -- (E5) -- (F5)  -- cycle;

  \coordinate (A6) at (1.33,-.33);
      \coordinate (B6) at (1.66,  .33);
      \coordinate (C6) at (2.33, .66);
      \coordinate (D6) at (2.66, .33);
      \coordinate (E6) at (2.33, -.33);
      \coordinate (F6) at (1.66, -.66);
      
      \draw[fill,lcred!100]  (A6) -- (B6) -- (C6) -- (D6) -- (E6) -- (F6)  -- cycle;

  \coordinate (A7) at (1.33,.66);
      \coordinate (B7) at (1.66,  1.33);
      \coordinate (C7) at (2.33, 1.66);
      \coordinate (D7) at (2.66, 1.33);
      \coordinate (E7) at (2.33, .66);
      \coordinate (F7) at (1.66, .33);
      
      \draw[fill,lcblue!100]  (A7) -- (B7) -- (C7) -- (D7) -- (E7) -- (F7)  -- cycle;

      \coordinate (A8) at (1.33,1.66);
      \coordinate (B8) at (1.66,  2.33);
      \coordinate (C8) at (2.33, 2.66);
      \coordinate (D8) at (2.66, 2.33);
      \coordinate (E8) at (2.33, 1.66);
      \coordinate (F8) at (1.66, 1.33);
      
      \draw[fill,lcbeige!100]  (A8) -- (B8) -- (C8) -- (D8) -- (E8) -- (F8)  -- cycle;

      \coordinate (0) at (0,0);
      \coordinate (1) at (1,0);
      \coordinate (2) at (0, 1);
      \coordinate (3) at (1, 1);
      \coordinate (4) at (2,0);
      \coordinate (5) at (0,2);
      \coordinate (6) at (1,2);
      \coordinate (7) at (2,1);
      \coordinate (8) at (2,2);

      \draw (-0.4,0)--(2.4,0);
      \draw (-0.4,1)--(2.4,1);
      \draw (-0.4,2)--(2.4,2);

      \draw (0,-0.4)--(0,2.4);
      \draw (1,-0.4)--(1,2.4);
      \draw (2,-0.4)--(2,2.4);

      \draw (-0.4,1.6)--(0.4,2.4);
      \draw (-0.4,0.6)--(1.4,2.4);
      \draw (-0.4,-0.4)--(2.4,2.4);
      \draw (0.6,-0.4)--(2.4,1.4);
      \draw (1.6,-0.4)--(2.4,0.4);


      \node[circle,inner sep=.6pt,fill=black,label={[scale=.47]above left:$(0,0)$}] at (0,0) {};
  \node[circle,inner sep=.6pt,fill=black,label={[scale=.47]above left:$(1,0)$}] at (1,0) {};
  \node[circle,inner sep=.6pt,fill=black,label={[scale=.47]above left:$(0,1)$}] at (0,1) {};
  \node[circle,inner sep=.6pt,fill=black,label={[scale=.47]above left:$(1,1)$}] at (1,1) {};
  \node[circle,inner sep=.6pt,fill=black,label={[scale=.47]above left:$(0,2)$}] at (0, 2) {};
  \node[circle,inner sep=.6pt,fill=black,label={[scale=.47]above left:$(2,0)$}] at (2, 0) {};
  \node[circle,inner sep=.6pt,fill=black,label={[scale=.47]above left:$(1,2)$}] at (1, 2) {};
  \node[circle,inner sep=.6pt,fill=black,label={[scale=.47]above left: $(2,1)$}] at (2,1) {};
  \node[circle,inner sep=.6pt,fill=black,label={[scale=.47]above left:$(2,2)$}] at (2, 2) {};
  \end{tikzpicture}

  \caption{The discrete torus graph $A_2/3A_2$ of the hexagonal
    lattice which has $9$ vertices and $27$ edges. This graph gives a
    certificate that $\chi(A_2) \leq 3$}
\end{figure}

\subsection{Proof of the main theorem}
\label{ssec:proof}

The proof of our main theorem ultimately involves a finite list of
computations, although this was not guaranteed apriori. The general strategy is as follows.  We separately
consider each of the 16 isomorphism classes of Cayley graphs of
four-dimensional lattices.  For each such graph $C$, we compute the
Delaunay polytope bound, that is the chromatic number of the subgraph
$C_1$ (see Section \ref{sec:dpb}).  Next we select a lattice $\Lambda$
representing $C$ and choose a suitable sublattice $\Lambda'$ such that
the discrete torus bound induced by $\Lambda'$ matches the Delaunay
polytope bound. It turns out, simply by trial and error, that the
sublattices $\Lambda'$ can be chosen to be of the form $c\Lambda$ for
some $c \in \N$, such that the bounds coincide. The information on
which lattice and constant $c$ we used for our computation can be
found in Table \ref{table:Cayley-graphs-computational-data}.

\begin{table}[htb]
  \begin{tabular}{lllllll}
    \hline
   \# & graph & $\Lambda$ & $c$ & $\chi$ \\
    \hline
    1 & $V_{30}^z$ & $K_5$ & 5 & 5\\
    2 &  $V_{30}^{z,n}$ & $K_{3,3}$ & 7 & 7\\
    3 &  $V_{30}^n$ & $111-$ & 6 & 6\\
    \hline
    4 &  $V_{28}^z$ & $K_5 - 1$ & 5 & 5\\
    5 &  $V_{28}^n$ & $211+$ & 6 & 6\\
      \hline
    6 &  $V_{26}^z$ & $K_5 - 1 - 1$ & 5 & 5\\
    7 &  $V_{26}^n$ & $221+$ & 6 & 6\\
      \hline
    8 &  $V_{24}^z$ & $K_5-2-1$ & 5 & 5\\
    9 &  $V_{24}^n$ & $222+$ & 4 & 4\\
      \hline
    10 &  $V_{22}^z$ & $C_{2221}$ & 4 & 4\\
    11 &  $V_{20}^z$ & $K_5 - 3$ & 5 & 5\\
    12 &  $V_{16}^z$ & $K_4+1$ & 4 & 4\\
    13 &  $V_{14}^z$ & $C_{221} + 1$ & 4 & 4\\
    14 &  $V_{12}^z$ & $C_3+C_3$ & 3 & 3\\
    15 &  $V_{10}^z$ & $C_3+1+1$ & 3 & 3\\
    16 &  $V_{8}^z$ & $1+1+1+1$ & 2 & 2\\
    \hline
  \end{tabular}
    
  \smallskip
    
  \caption{Computational data used to compute the chromatic number of
    the 16 cases of inequivalent Cayley graphs associated to $4$
    dimensional lattices: For each graph we specify which lattice
    $\Lambda$ and which constant $c$ we used to compute the discrete
    torus bound associated to $\Lambda/c\Lambda$. $\chi$ gives the
    chromatic number of the graph, as obtained by verifying that the
    Delaunay polytope bound and discrete torus bound coincide.}

  \label{table:Cayley-graphs-computational-data}
  
  \end{table}

  We also quickly recall how to compute the chromatic number of a
  finite graph by transforming the problem into a Boolean
  satisfiability problem, so that a standard SAT solver can be used.

  So, let $G = (V,E)$ be a finite graph. Since a coloring of $G$ is a
  partition of the vertex set into independent sets (the color
  classes) we have to model:
\begin{enumerate}
  \item If $v \in V$, then $v$ has a color.
  \item If $\{v,w\}\in E$, then $v$ and $w$ do not have the same color.
\end{enumerate}
Note that in principle these two clauses would allow a vertex $v$ to have multiple colors, we do not care about this, as we can then always extract a proper coloring by choosing one of the colors for each multi-colored vertex.

We look for the smallest number $k$ such that there exists a coloring
with $k$ color classes.  This can be rephrased as a Boolean formula in
conjunctive normal form (CNF) as follows. Write
$V = \{v_1,\ldots,v_n\}$; these are the variables. Assume that the
colors are given by $1,\ldots,k$. We use a total of $n\cdot k$
variables
\[
  \{x_{i,\ell}\}_{\substack{i=1,\ldots,n\\\ell=1,\ldots,k}}
\] 
where the value of $x_{i,\ell}$ is to be interpreted by
\[
 x_{i,\ell} = \begin{cases} 1, & \text{vertex $v_i$ has
      color $\ell$} \\ 0, & \text{else.} \end{cases}
\]
With these variables we formulate a SAT problem with two types of clauses:
\begin{enumerate}
  \item For every vertex $v_i$ the clause $x_{i,1} \vee \cdots \vee x_{i,k}$.
  \item For each pair of adjacent vertices $v_i,v_j$ the clause $\bigwedge_{l=1}^k \neg x_{i,\ell} \vee \neg x_{j,\ell}$.
\end{enumerate}
Now the chromatic number is the smallest $k$ for which the SAT problem
above is satisfiable.

\bigskip

We used the MAGMA program \texttt{chi-dimension-4.magma} to perform the
necessary computations and to generate the input for the SAT solver,
for which we used PySAT \cite{PySAT-2018}. This MAGMA program can be found as an
ancillary file from the \texttt{arXiv.org} e-print archive. The MAGMA
program contains the following:
\begin{enumerate}
\item The $52$ lattices associated to the $52$ inequivalent Delaunay
  subdivisions in dimension $4$. These are the lattices contained in
  Table \ref{table:zonotopal-cases} and Table \ref{table:nonzonotopal-cases};
\item A verification of Table \ref{table:Cayley-graphs}: of the $52$
  Cayley graphs associated to these lattices only $16$ are pairwise
  non-isomorphic;
\item The generation of SAT problems to confirm the Delaunay polytope
  bound and discrete torus bound. For each graph case
  we provide a number $c$ (which turns out to satisfy $c= \chi$) and
  \begin{enumerate}
    \item a certificate that $C_1$ is not $c-1$-colorable (DPB);
    \item a certificate that the associated discrete torus $\Lambda/c\Lambda$ is $c$ colorable (DTB).
  \end{enumerate}
\end{enumerate}
Here, we stick to the notation from Section \ref{sec:bounds} and Table
\ref{table:Cayley-graphs-computational-data}.

\section{Dimension 5 and onwards}
\label{sec:dimension-5}

We conclude the present discussion by some observations and open
questions.

Firstly, the complete classification of inequivalent Delaunay
subdivisions is also known in dimension $5$ by 
Dutour Sikiri\'c, Garber, Sch\"urmann, and Waldmann \cite{Dutour-2016}. So our brute-force
computational approach to find the chromatic number of all
$5$-dimensional lattices is, in principle, applicable.  However, all
our computations for dimension $5$ remained inconclusive: The size of
the finite graphs that need to be considered grows drastically, in
any case where a SAT solver could provide an answer in reasonable
time, the bounds were far apart.  We think that strengthening the
bounds is a more promising strategy than throwing a massive amount of
computational time on the problem.  For this reason we restricted to
the $4$-dimensional case for now.

So, some questions remain:
\begin{enumerate}
\item Can we strengthen the bounds presented here, or find a way to
  compute them faster?
\item \label{question-reviewer} The maximal chromatic number in dimension 4 is 7, while the best known theoretical upper bound is $2^4$. Can one prove that indeed $2^{n-1}$ is a valid upper bound on the chromatic number of $n$-dimensional lattices for $n \geq 3$? It is true in dimension $3$ by the results in \cite{Dutour-2021}.
\item For all lattices studied in this work and in \cite{Dutour-2021} we are aware of at least one periodic optimal coloring (while non-periodic optimal colorings may also exist). Now, given a general lattice, is it always true that there exists a periodic optimal coloring?
\item \label{question-periodic} More generally, if $\Lambda'$ is a sublattice of $\Lambda$, what
  can we say about $\Lambda'$-periodic (optimal) colorings of $\Lambda$? Do they
  exist? If so, how many such? How small can the index $[\Lambda: \Lambda']$ become? Is there a bound $B(n)$ (depending on the dimension) such that there always exists a $\Lambda'$-periodic optimal coloring with $[\Lambda:\Lambda'] \leq B(n)$. In all known cases $[\Lambda : \Lambda'] \leq \chi(\Lambda)^n$ is true.
\end{enumerate}

\section*{Acknowledgments}

The first named author likes to thank Mathieu Dutour Sikiri\'c for
helpful discussions. 

\smallskip

The authors also express their gratitude to the suggestions of one of the anonymous referees. In comparison to an earlier draft of this manuscript we included a suggested strengthening of the statement of Lemma \ref{lem:dtb}, adapted question \eqref{question-reviewer} from the referee report, and revised question \eqref{question-periodic} concerning periodic optimal colorings.

\end{document}